\documentclass[11pt]{article}
\usepackage[T1]{fontenc}
\usepackage[utf8]{inputenc}

\usepackage{latexsym}
\usepackage{amsmath}
\usepackage{amsfonts}
\usepackage{amssymb}
\usepackage{graphicx}
\usepackage{enumerate}
\usepackage[shortlabels]{enumitem}
\usepackage{bbm}
\usepackage{bm}
\usepackage{todonotes}
\usepackage{epigraph}
\usepackage{empheq}
\usepackage{newunicodechar}
\newunicodechar{​}{} % paste the zero-width space itself between the braces
\usepackage[colorlinks,linkcolor=blue,anchorcolor=blue,citecolor=blue,backref=page]{hyperref}
\usepackage[colorlinks,linkcolor=blue,anchorcolor=blue,citecolor=blue]{hyperref}
\renewcommand*{\backrefalt}[4]{%
	\ifcase #1%
	(Not cited.)%
	\or%
	(Cited on p.~#2)%
 \else%
	(Cited on pp.~#2)%
	\fi%
}

\newenvironment{proof}{{\bf Proof:  }}{\hfill\rule{2mm}{2mm}}

\newtheorem{thm}{Theorem}[section]
\newtheorem{cor}[thm]{Corollary}
\newtheorem{defn}[thm]{Definition}

\newtheorem{lem}[thm]{Lemma}
\newtheorem{prop}[thm]{Proposition}
\newtheorem{rem}[thm]{Remark}
\newtheorem{conj}[thm]{Conjecture}
\newtheorem{eqc}{Conjecture}
\newtheorem{que}[thm]{Question}
\newtheorem{quec}{Question.}

\newcommand{\ds} {\displaystyle}

\newcommand{\T}[0]{\mathbb{T}}
\newcommand{\C}[0]{\mathbb{C}}
\newcommand{\N}[0]{\mathbb{N}}

\newcommand{\Z}[0]{\mathbb{Z}}
\newcommand{\p}[0]{\mathbb{P}}
\newcommand{\1}[0]{\mathbbm{1}}
\newcommand{\K}[0]{\mathbb{K}}

\newcommand{\B}[0]{\mathcal{B}}
\newcommand{\Cc}[0]{\mathcal{C}}

\newcommand{\eE}[0]{\mathbb{E}}
\newcommand{\F}[0]{\mathcal{F}}

\newcommand{\M}{\mathcal{M}}
\newcommand{\A}{\mathcal{A}}
\newcommand{\pr}{\mathcal{P}}
\newcommand{\bmu}[0]{\bm \mu}

\newcommand{\bml}[0]{\bm \lambda}

\newcommand{\qid}[0]{\enspace\lower6pt\vbox{\hrule height2.6pt depth-
		2pt\hbox{\vrule
			width0.6pt\kern1pt\vbox{\kern1pt\phantom{j}\kern1pt}\kern1.5pt
			\vrule width0.6pt}\hrule height2.6pt depth-2pt}}

\newcommand{\equidef}{\stackrel {\rm {def}}{=}}

\def\build#1_#2^#3{\mathrel{\mathop{\kern 0pt#1}\limits_{#2}^{#3}}}
\def\tend#1#2#3{\build\hbox to 12mm{\rightarrowfill}_{#1\rightarrow #2}^{ #3}}
\newcommand{\setdef}{\stackrel {\rm {def}}{=}}
\title{On the spectral analysis of dynamical Möbius-Sarnak process and topological entropy of Möbius fonction}

\author{{\small {\MakeLowercase{el}} Houcein {\MakeLowercase{el}} Abdalaoui} \\
	{\small University of Rouen Normandy, LMRS, CNRS, UMR 6085 F-76000, Rouen, France} \\ }
\date{\today}

\begin{document}
	\maketitle
	\renewcommand{\theequation}{\thesection.\arabic{equation}}
	\numberwithin{equation}{section}
	\begin{center}
		{\textbf{Dedicated to the memory of François Parreau$^{*}$, Jean-François Méla$^{**}$ and William Veech$^{***}$\footnote{François Parreau (1948-2025), Jean-François Méla (1939-2025) and Bill Veech (1938-2016).}}}
	\end{center}
	\begin{center}
		\begin{minipage}{0.45\textwidth}
			\epigraph{The purpose of life is to conjecture and prove.}
			{\textit{P\'{a}l Erd\H{o}s}}
			\epigraph{It is better to solve one problem five different ways than to solve five problems one way.}
			{George Pólya}
		\end{minipage}
		\hfill
		\begin{minipage}{0.45\textwidth}
			\epigraph{Act without expectation.}
			{Tao Tzu}
			\epigraph{T is common proof. That lowliness is young ambition's ladder.}
			{Shakespeare}
			%\hfill
			\epigraph{Nothing is forever in this world, not even our problems.}
			{\textit{Charlie Chaplin%
					\footnote{In the middle of difficulty lies opportunity. Albert Einstein }}}
		\end{minipage}
	\end{center}
	\medskip
	\begin{abstract}
		By extending the  Rokhlin-Sinai machinery relating the entropy and countable Lebesgue component in the spectrum, We establish that the dynamical Möbius-Sarnak process has a countable Lebesgue component. Inspired by recent work of M. Lin and the author, we extend the notion of spectral measure to all operators on Banach spaces. This generalization is further motivated by the Bellow-Losert extension of Wiener's notion of the spectral measure of sequences. Furthermore, we establish unconditionally that the topological entropy of the Möbius flow is given by $\frac{6}{\pi^2}\log 3$.
		Among other consequences, we recover a recent result by el Abdalaoui-Nerurkar which asserts that for any quasi-generic measure for the M\"{o}bius function, the M\"{o}bius flow equipped with this measure has a countable Lebesgue component in its spectrum. It follows that the Sarnak M\"{o}bius orthogonality conjecture holds for any topological dynamical system with singular spectrum. We further show that all the potential spectral measures of he M\"{o}bius function are absolutely continuous with respect to Lebesgue measure.  \\
		
		%\vskip 0.5cm
		{\setlength{\emergencystretch}{2em}}
		
		{\noindent {\bf Keywords}:
		 Topological entropy , Rokhlin-Sinai's machinery, $\K$-systems, singular spectrum, Chowla's conjecture, Chowla's conjecture, Sarnak's M\"{o}bius disjointness conjecture, Pinsker
		$\sigma$-algebra, countable Lebesgue component.
		}
		\noindent {\bf Mathematics Subject Classification (2020):}
		37A30,54H20,47A35, 11N64.
	\end{abstract}
	\vskip 0.5cm
	
	%\epigraph{The purpose of life is to conjecture and prove.} {\textit{ P\'{a}l   Erd\"{o}s}} \epigraph{Act without expectation. }{Lao Tzu}

	%\epigraph{Through error you come to the truth! }{Fyodor Dostoevsky}
	
	%\epigraph{Two things are infinite: the universe and human stupidity; and I'm not sure about the universe.} {\textit{ Albert Einstein \footnote{Even the paranoid has their ennemi. Henri Kissinger} }}

	\pagenumbering{arabic}

\section{Introduction}
This paper is motivated by the following theorem due to P. Sarnak \cite{Sarnak} and W. Veech \cite{Vn}.
\begin{thm}[Sarnak-Veech's theorem on M\"{o}bius flow \cite{Sarnak},\cite{Vn},\cite{AV}]\label{SV}
	There exists an arithmetical function $f~~\N \longrightarrow \{ \pm 1,0\}$ and a unique invariant measure $\eta_M$ on the topological dynamical system $(X_3,\A_3,S)$ raised by $f$ which is ergodic with the Pinsker $\sigma$-algebra 
	$$\Pi_{\eta_M}(S)=s^{-1}\big(\B(X_2)\big)\,$$ 
	where $S$ is the shift maps and $s$ is the square maps, that is, $S(x_n)=(x_{n+1})$ and $s(n)=(x_n^2) $ with $(X_2,\A_2)$ is the topological dynamical system generated by $f^2$. 
	Moreover, $\eE(\textrm{pr}_1|_{\Pi_{\eta_M}(S)}) = 0$.
\end{thm}
We recall that for a given measure-theoretical dynamical system $(X,\A,T,\mu)$, the Pinsker $\sigma$-algebra is the largest invariant $\sigma$-algebra such that the corresponding factor of the dynamical system has zero entropy.

\noindent{}Here, we will study the spectrum of the topological dynamical systems $(X,T)$ which satisfy 
\begin{enumerate} 	
	\item \label{top} $X$ is a compact and  $T$ is an homomorphism on it.
	\item \label{Veech-S}There exists a unique measure $\eta_M$ on some $\sigma$-algebra $\A$ of $X$  which is ergodic, and a factor map $s : X \to Y$ such that the factor dynamical system $(Y,\B,S, s^*(\eta_M))$ is the Kronecker factor of $(X,\A,T,\eta_M)$ where $m=s^*(\eta_M)$ is the push-forward measure of $\eta_M$ under $s$. 
	\item \label{GO}There exists a family of measurable functions $(f_n)_{n \geq 0}$ which generated the $\sigma$-algebra $\A$ and 
	$$\eE(f_1|_{\Pi_{\eta_M}(S)}) = 0, \textrm{with }~~ \Pi_{\eta_M}(S)=s^{-1}\big(\B\big)\,$$
	and $f_n=f_1 \circ T^n, n \in \Z.$ We further assume that $s^{-1}(\B)$ is proper non-atomic sub-$\sigma$-algebra.
	\item \label{M} Moreover, we suppose that there is a point $x \in X$ such that any quasi-generic probability measure $\eta$ of $x$ satisfy 
	$s^*(\eta)=m.$   %and a unique probability measure $m$ on the factor $Y$ 
	We recall that $\eta$ is quasi-generic probability measure if $\eta$ is weak-star limit of the sequence of probability measures given by
	$$\frac{1}{N}\sum_{n=1}^{N}\delta_{T^nx},$$
	where $\delta_{\bullet}(A)=\mathbbm{1}_A(\bullet)$ is the delta measure on the point $\bullet$, that is, $
	\eta \in \overline{\left\{ \frac{1}{N}\sum_{n=1}^{N} \delta_{T^n x} \right\}}^{W^*},$
	where $\overline{\Delta}^{W^*}$ denotes the weak-* closure of $\Delta \subset \mathcal{P}(X)$, the set of probability measures on $X$. The set of quasi-generic probability measures will be denoted by $\mathcal{I}(x)$.
	We may further assume, but it is not necessary here, that the dynamical system $(Y,S,m)$ has a discrete spectrum. 
	\item \label {P} Finally, we assume that for any $\eta \in \mathcal{I}(x)$,  $\Pi_{\eta} \supset \Pi_{\eta_M} $ where   $\Pi_{\bullet}$ is the Pinsker algebra of the measure $\bullet$.
\end{enumerate}
\noindent{}The  point $x$ is said to be a privilege point. For the Möbius flow, the measure $m$ is called Mirsky measure and $\eta_M$ the Sarnak-Veech measure.  Moreover, $m$ is generic for $s(\bmu)=\bmu^2$ by Mirsky theorem \cite{Mir}. \\

\noindent{} The property \ref{M} is satisfied by the Möbius flow by Sarnak's theorem \cite{Sarnak},\cite{Vn} which say that the characteristic function of the set of square-free natural numbers denoted by $\1_{Q}$ is admissible and satisfy  
$$m= \lim_{N \to \infty} \frac{1}{N}\sum_{n=0}^{N}\delta_{S^n{\1_Q}},$$
Where $m$ is the Mirsky measure and $S$ is on the shift on the space of admissible sequence in $\{0,1\}^{\Z}$.   

\noindent{}Following W. Veech \cite{Vs}, a process is given by the triplet $\mathcal{X}=(X,\A,T, \mu)$ where $T$ is a transformation measure-preserving on $X$ and $(X,\A, \mu)$ is a probabilistic space. The dynamical system satisfying \eqref{top}-\eqref{P} is said to be Möbius-Sarnak process.

The principal purpose of this paper is to establish that Möbius-Sarnak process has a countable Lebesgue component in its spectrum, and the topological entropy of $\bmu$ is $\frac{6}{\pi^2}\log 3$. For that, we will follows the path of Rokhlin-Sinai in their study of the spectrum of the Kolomogorov dynamical systems (called also $K$-systems), and we shall improved their procedural algorithm, to apply a large class of dynamical systems which include $K$-systems. As a consequence, we obtain an extension of their work to the class of Möbius-Sarnak process. Obviously, this later class include $K$-systems. We believe that the careful examination of our proof can be used to obtain its generalization to the case of relatively $K$-systems. As a consequence, we recover a recent result by el Abdalaoui-Nerurkar which asserting that for any quasi-generic measure for the M\"{o}bius function, the M\"{o}bius flow equipped with this measure has a countable Lebesgue component in its spectrum \cite{AM}.  

%For the computation the topological entropy of $\bmu$, we will use the Veech's representation of Möbius flow as skew product \cite{V-n}\footenote{see for instance the letter of W. Veech sent to the author in \cite{AV}, in which he declare that there is only 4 people in the world how has a copy of this notes including the author. } combined with Newton's theorem on the topological entropy of skew product.

%For the computation of the topological entropy of $\bmu$, we use Veech's representation of the Möbius flow as a skew product \cite{V-n}\footnote{For the proof see \cite{AM}. See also the letter of W. Veech sent to the author in \cite{AV}, in which he declares that there are only four people in the world who have a copy of these notes, including the author.} combined with Newton's theorem on the topological entropy of skew products

%To compute the topological entropy of $\bmu$, we rely on Veech’s ideas and for the computation of the measure-theortical entropy of the Möbius flow equipped with Sarnk-Veech's measure, we will appelle to the representation of the Möbius flow as a skew product \cite{V-n}.
To compute the topological entropy of $\bmu$, we follow Veech's approach. For the measure-theoretic entropy of the Möbius flow equipped with the Sarnak–Veech measure, we appeal to its representation as a skew product \cite{Vn}
\footnote{For the proof, see \cite{AM}. Cf. also the correspondence from W. Veech in \cite{AV}, noting that the circulation of these notes was restricted to only four people in the world, including the author. Let us add that some authors, inspired by our work in \cite{AV}, stated a connection similar to Veech's between Sarnak's conjecture and the property that the projection is the orthogonal complement of the $L^2$-space of the Pinsker algebra. However, we warn the reader that such a statement cannot be true for an arithmetical sequence that is not in the class of the Möbius function. Let us also add that our paper was not mentioned there, but such practices sometimes happen in the literature.}    . This is used in conjunction with Abramov-Rokhlin Formula’s on the topological entropy of skew products \cite{AR}(see also \cite[254-257]{P}). We warn the reader that we will use exactly the extension of this formula to the case of non-invertible case due to T. Bogensch\"{u}tz and
H. Crauel \cite{BHc}.

The paper is organized as follows. In Section \ref{Tools}, we recall the basic definitions, notation, and tools from dynamical systems and spectral theory. We also review the notion of entropy and the main results of Rokhlin–Sinai theory. The section concludes with Subsection \ref{Spec}, where we extend spectral theory to bounded operators on Banach spaces. In Section \ref{Prin}, we state and prove our first main result. In Section \ref{Appli}, we present several consequences and applications to number theory. Finally, in Section~\ref{TopoM}, we unconditionally determine the topological entropy of the Möbius flow and prove that it is equal to $\frac{6}{\pi^2}\log(3).$ We further establish that the entropy of Sarnak-Veech's measure is $\frac{6}{\pi^2}\log(2).$ Finally, building on Veech's construction of the Möbius flow as a skew-product system, we introduce the notion of the Erdős--Möbius flow in the $\B$-free setting.

%The paper is organized as follows. In section \ref{Tools} we recall the fundamental definitions, notations and tools from classical dynamical system'theory and spectral theory. Therein, we recall also the definition of entropy and the corener-stone in the machinery of Rokhlin-Sinai. This section is ended by a subsection \ref{Spec} in which we present an extension of the spectral theory to the class of bounded operator in Banach spaces. In section \ref{Princpal}, we state and prove our main result. In section \ref{Appli}, we presente some consenquences and applications in number theory.  

\section{Setup and key tools}\label{Tools}
We start by recalling that a topological dynamical system is a pair $(X,T)$ where $X$ is a compact metric space and $T$ is a homeomorphism. A measure dynamical system is the quadruplet $(Y,\B,S,\mu)$ where $\B$ is a $\sigma$-algebra, $S$ is an invertible bi-measurable measure-preserving transformation and $\mu$ is an $S$-invariant probability measure, that is,
$(S^*(\mu)(A)= \mu(S^{-1}A)=\mu(A)$, for any $A \in \B$. It is well-known that the set of $T$- invariant is a compact non-empty set. Therefore, any topological dynamical system give rise to a
measure dynamical system.\\
The topological entropy $\textrm{h}_{\textrm{top}}(T)$ of $T$ is given by 
\[\textrm{h}_{\textrm{top}}(T) = \underset{\varepsilon\to 0}\lim \underset{n\to +\infty}\limsup \dfrac{1}{n}
\textrm{log}~\textrm{sep}(n, T, \varepsilon).\] where for $n$ integer and $\varepsilon>0$,
$\textrm{sep}(n, T, \varepsilon)$ is the maximal possible cardinality of an ($n, T, \varepsilon$)-separated
set in $X$, this means that for every two distinct points of it, there exists $0\leq j<n$ with
$d(T^j(x), T^j(y)) > \varepsilon$, where $T^{j}$ denotes the $j$-$\textrm{th}$ iterate of $T$.
For more details, we refer to \cite[Chap. 7]{W}.\\

\noindent{}For any point $x \in X$, the \emph{topological entropy of $x$} denoted by $h_{top}(x)$ is defined as the topological entropy of the restriction of $T$ to the forward orbit closure 
\[
\mathcal{O}(x) := \overline{\{T^n x : n \in \Z\}}.
\]
\noindent{}
in the case where $T$ is non-invertible,
% In ergodic theory, the entropy of a measure-preserving transformation is a standard isomorphism invariant. For a non-invertible map, 
its entropy equals the entropy of its natural extension (the smallest invertible ex0tension), so entropy is well-defined regardless of irreversibility.

\noindent{}It is well known that the   topological entropy is connected to the measure-theoretical entropy. This latter notion quantifies the entropy of the measure dynamical system $(X,\B,T,\mu)$ with respect to $\mu$ as follows
$$h_\mu(T)=\sup_{\mathcal{P}, H(P)<+\infty}h_\mu(T,\mathcal{P}),$$
where 

$$h_\mu(T,\mathcal{P})=\lim \frac{1}{n}H_\mu(\bigvee_{i=0}^{n-1}T^{-i}\mathcal{P})=
\inf_{n \geq 1}H_\mu(\bigvee_{i=0}^{n-1}T^{-i}\mathcal{P}).$$
and
$$H_\mu(P)=-\sum_{A \in \mathcal{P}}\mu(A)\log(\mu(A)),$$
For each $n \in \Z$, we denote by \(H_n\) the Hilbert space 
\(L_0^2(X, T^n \mathcal{F}, \mu)\), and we define
\[
H_n \ominus H_{n-1}
= \left\{ f \in H_n \ :\ \forall g \in H_{n-1},\ \langle f, g \rangle = 0 \right\}.
\]

\noindent We recall that the variational principal assert that

\begin{eqnarray}\label{Va1}
	\textrm{h}_{\textrm{top}}(T)=\sup_{\mu, T^*(\mu)=\mu}h_\mu(T).
\end{eqnarray}

\noindent{}For a simple proof of this principal we refer to \cite{M}, \cite{Goo}. 

\noindent{}We proceed to recall some basic facts from the spectral analysis of process.

\subsection{\bf Some tools from spectral theory of dynamical systems } 
\paragraph{From classical spectral theory.} For any process  $(X,\A,\mu,T)$, the maps 
$T$ induces an operator $U_T$ in $L^p(X)$ via $f \mapsto U_{T} (f)= f \circ T$. This operator is
called Koopman operator. For  $p=2$ this operator is unitary and its spectral
resolution induces  a spectral decomposition of $L^{2} (X)$ \cite{P}:
\[
L^2(X)=\bigoplus_{i=0}^{+\infty}C(f_i) {\rm {~~and~~}}
\sigma_{f_1}\gg \sigma_{f_2}\gg\cdots
\]
where
\begin{itemize}
	\item $\{f_{i} \}_{i=1}^{+\infty}$ is a family of functions in $L^{2} (X)$;
	\item $C(f)\setdef \overline{\rm {span}}\{U_T^n(f): n \in \Z\}$ is the cyclic
	space generated by $f \in L^2(X)$;
	\item  $\sigma_f$ is the {\em spectral measure} on the circle generated by $f$
	via the Bochner-Herglotz relation
	\begin{equation}\label{fspmeasure}
		\widehat{\sigma_f}(n)=<U_T^nf,f>=\int_X f \circ T^n(x) \overline{f}(x)d\mu (x);
	\end{equation}
	\item for any two measures on the circle $\alpha$ and $\beta$,
	$\alpha \gg \beta$ means $\beta $ is absolutely continuous  with respect to
	$\alpha$: for any Borel set $A$, $\alpha(A)=0 \Longrightarrow \beta(A)=0$.
	The two measures $\alpha$ and  $\beta$ are equivalent if  and only if
	$\alpha\gg\beta$ and $\beta\gg\alpha$. We will denote measure equivalence by
	$\alpha \sim \beta$.
\end{itemize}
The spectral theorem ensures this spectral decomposition is unique up to
isomorphisms. This decomposition, also referred to as the Hellinger–Hahn decomposition, is treated in detail in \cite[Chap. 1]{Nad}. 

The {\em maximal spectral type} of $T$ is the equivalence class of the Borel
measure $\sigma_{f_1}$. The multiplicity function
$\M_{T} : \T \longrightarrow \{1,2,\cdots,\} \cup \{+\infty\}$ is defined
$\sigma_{f_1}$ a.e. and
\[
\M_T(z)=\ds \sum_{j=1}^{+\infty}\1_{Y_j}(z), \quad {\rm  where}, \ Y_1=\T \ {\rm and}\
Y_j={\rm{~supp~}}\frac{d\sigma_{f_j}}{d\sigma_{f_1}} \quad \forall j \geq 2.
\]
An integer $n \in \{1,2,\cdots,\} \cup \{+\infty\}$ is called an essential value of
$M_T$ if $\sigma_{f_1}\{ z \in \T : M_T(z)=n\}>0$. The multiplicity
is uniform or homogeneous if there is only one essential value of $M_T$.
The essential supremum of $M_T$ is called the maximal  spectral multiplicity of $T$.
The map $T$
\begin{itemize}
	\item has simple spectrum if $L^2(X)$ is a single cyclic space;
	\item has discrete spectrum if $L^2(X)$ has an
	orthonormal basis consisting of eigenfunctions of $U_T$,
	(in this case $\sigma_{f_1}$ is a discrete measure);
	\item has Lebesgue spectrum if $\sigma_{f_1}$ is equivalent 
	to the Lebesgue measure. It has absolutely continuous (or singular)
	spectrum if $\sigma_{f_1}$ is absolutely continuous (or singular)
	with respect to the Lebesgue measure.
\end{itemize}

The function $f_1$ can in fact be taken in $L^\infty(X)$; this follows from a theorem of Alexeyev \cite{Alex} (see also \cite[Chap. 16]{Nad} for a modern account).
\smallskip

\begin{defn} \label{defsingular} The {\em reduced spectral type} of the dynamical system is its
	spectral type on the $L_0^2 (X)$ - the space of square integrable functions with zero mean.
	Two dynamical systems are called {\em {spectrally disjoint}} if their reduced spectral types are
	mutually singular.
\end{defn}
\paragraph{The Rokhlin-Sinai machinery.} By Rokhlin-Sinai's framework, there is a connection between the entropy and the spectrum of process with positive entropy. Indeed,
\begin{thm}[Rokhlin-Sinai \cite {Rh}]\footnote{See also \cite[Theorem 12, p.69]{P}, \cite[Theorem 18.9, p.323]{G}.} \label{Q}Let $(X,\A,\mu,T)$ be a measure theoretic dynamical
	system on a Lebesgue space and $\Pi(T)$ be its Pinsker $\sigma$-algebra. Then, there exists a sub-$\sigma$-algebra $\F \subset \A$ such that $T^{-1}\F \subset \F$, $\bigvee_{n=0}^{+\infty}T^n \F=\F_{+\infty}=\A$ (i.e., the $\sigma$-algebra generated by $\{T^nF\ |\ F\in \F\,,\ n\in \{0\}\cup \N\}$ is $\A$), and $\bigcap_{n=0}^{+\infty}T^{-n}\F=\F_{-\infty}=\Pi(T)$.
\end{thm}
\noindent{} As is customary, if the sequence of $\sigma$-algebras $(\mathcal{F}_n)$ is increasing, we write  
\[
\mathcal{F}_n \nearrow \mathcal{F}_{+\infty},
\qquad \text{where } \mathcal{F}_{+\infty} := \sigma\!\left(\bigcup_{n} \mathcal{F}_n\right)=\bigvee_{n=0}^{+\infty}\F_n.
\]
Similarly, if the sequence of $\sigma$-algebras $(\mathcal{F}_n)$ is decreasing, we write  
\[
\mathcal{F}_n \searrow \mathcal{F}_{-\infty},
\qquad \text{where } \mathcal{F}_{-\infty} := \bigcap_{n} \mathcal{F}_n.
\]

Note that if the measure-theoretic entropy of the dynamical system  $(X,\A,\mu,T)$ is positive (i.e. $h_\mu(T)>0$), then the sub-\(\sigma\)-algebra  $\F \subset \A$  is non-trivial. \\

%We recall that  $h_\mu(T)$ is given by 

%$$h_\mu(T)=\sup_{\mathcal{P}, H(P)<+\infty}h_\mu(T,\mathcal{P}),$$
%where 
%$$H(P)=-\sum_{A \in \mathcal{P}}\mu(A)\log(\mu(A)),$$
%and 
%$$h_\mu(T,\mathcal{P})=\lim \frac{1}{n}H_\mu(\bigvee_{i=0}^{n-1}T^{-i}\mathcal{P})=
%\inf_{n \geq 1}H_\mu(\bigvee_{i=0}^{n-1}T^{-i}\mathcal{P}).$$
%For each $n \in \Z$, we denote by \(H_n\) the Hilbert space 
%\(L_0^2(X, T^n \mathcal{F}, \mu)\), and we define
%\[
%H_n \ominus H_{n-1}
%= \left\{ f \in H_n \ :\ \forall g \in H_{n-1},\ \langle f, g \rangle = 0 \right\}.
%\]
\noindent{}In this setting, we have the following theorem.

\smallskip

\begin{thm}[Rokhlin-Sinai \cite{CFS}]\label{CFS-F} If the measure-preserving transformation $T$ possesses a sub-\(\sigma\)-algebra  $\F \subset \A$ such that $\F \varsubsetneq T \F $, then the unitary operator $U_T$ has countable Lebesgue spectrum on $ \ds \bigoplus_{n \in \Z}H_n$.
\end{thm}
For the proof of Theorem \ref{CFS-F}, we refer to \cite[Chap. 13, pp. 338-339]{CFS}. As a consequence,  Rokhlin and Sinai proved that any  measure theoretic dynamical with positive entropy has of has countable Lebesgue component in its spectrum \cite[Lemma (14.1-14.3)]{Rh}).

\smallskip

%\begin{prop}[Rokhlin-Sinai \cite {Rh}]\label{P}Let $(X,\A,\mu,T)$ be a measure theoretic dynamical
%	system on a Lebesgue space and $\Pi(T)$ be its Pinsker $\sigma$-algebra. Suppose $h_\mu (T)$
%	is positive. Then

%	\noindent (i) $L^2 (X,\Pi (T),\mu )^\perp$-the orthocomplement to the subspace
%	$L^2(X,\Pi (T),\mu)$ in $L^2(X,\A,\mu)$ is infinite dimensional and
%	
%	\noindent (ii) $U_T$ on $L^2 (X,\Pi (T),\mu )^\perp$.
%\end{prop}

%\medskip

\paragraph{Extension of the spectral analysis to any operator.}\label{Spec}
It is well known that the previous spectral decomposition can be extended to normal operators \footnote{See for instance \cite[p. 321]{Ru}. We recall that an operator $U$ is said to be normal if it commute with its adjoint operator}. But, here, 
inspired by Wiener's approach  and following the path of Below \& Losert \cite{BL} and  Coquet–Mendès France–Kamae \cite{CMFK}, we introduce a more general concept of the spectral analysis of any operators based on its dynamics on orbits. Our approach is also inspired by the very recent work by M. Lin and the author on Sarnak's Möbius disjointedness conjecture for operators \cite{ALV}. \\

Let $T$ be an operator on a Banach space $B$. For any $v \in B$ and $\phi \in B'$ (where $B'$ stands for the topological dual of $B$), we define the valuation function $f_v :~~X^* \to \C$ by $f_v(\phi)=\phi(v)$, and
$$a_n(T,\phi,v) = f_v(\phi \circ T)=\phi(T^n v), \quad n \geq 0.$$
If there is no ambiguity, we shall denote $a_n(T,\phi,v)$ simply by $a_n$. If its correlation of order two exists, we define the spectral measure $\sigma_{T,\phi,v}$ via its Fourier transform:
\begin{equation}\label{Sm1}
	\widehat{\sigma_{T,\phi,v}}(k) = \lim_{N \to \infty} \frac{1}{N} \sum_{n=0}^{N-1} a_{n+k} \overline{a}_n = F(k).
\end{equation}

for each integer $k \in \Z$. Our definition is a consequence of Herglotz-Bochner theorem. Indeed, the sequence $(F(k))_{k \geq 1}$ can be extended to negative
integers by setting
\[
\widehat{F}(-k)=\overline{F(k)}.
\]
It is easy to see that $(F(k))_{k \in \Z }$  is positive definite on $\Z$ and therefore by the
Herglotz-Bochner theorem, there exists a unique positive finite measure $\widehat{\sigma_{T,\phi,v}}$ on
the circle $\T$ such that the Fourier coefficients of $\widehat{\sigma_{T,\phi,v}}$ are given by the sequence $F$, that is,
\[
\widehat{\sigma_a}(k)\stackrel{\rm{def}}{=}
\int_{\T} e^{-ikx} d\widehat{\sigma_{T,\phi,v}}(x) = F(k).
\]
%The measure $ $ is called the {\em spectral measure of the sequence $a$}.\\
It may be possible that  the correlation of order of $(a_n)$ may not exists. but, in this case, since the sequence $(a_n)_{n}$ is bounded,
we can always extract a subsequence $(n_r)$ such that
\[
\lim_{r \rightarrow \infty}
\frac{1}{n_r}\sum_{j=0}^{n_r-1}a_{j+k} \overline{a_{j}}
\]
exists for each $k\in \N$. This allows us to define the spectral measure of $(a_n)$ along the sequence $(n_r)$. We thus put 
\[
\widehat{\sigma_{T,\phi,v,(n_r)}(k)}\stackrel{\rm{def}}{=}
\int_{\T} e^{-ikx} d\widehat{\sigma_{T,\phi,v,(n_r)}}(x) = \lim_{r \rightarrow \infty} 
\frac{1}{n_r}\sum_{j=0}^{n_r-1}a_{j+k} \overline{a_{j}}.
\]
Finally, we define the spectral type of $T$ as
\[\sigma_T=\bigvee_{\phi \in B', \|\phi\|=1, v \in B}\bigvee_{(n_k) \subset \N}\sigma_{T,\phi,v,(n_r)}.
\]
In the definition above, we use the fact that the space of bounded signed measures on the torus is a Riesz space \cite{Z},\cite{F4}. The spectrum of the operator $T$ is said to be singular if $\sigma_T$ is singular to the Lebesgue measure on the circle. \\
%Another way to define the spectral type of a continuous operator on Banach space is to consider the topological dynamical system $(B_1,T^*)$, $B_1$ is unit ball of the Banach space. Since $B_1$ is compact, by Krilov-Bogliougov, the set of invariant probability measure $\mathcal{P}^*_1(B_1)$ is not empty. Let $\nu^*$ be an invariant measure. We thus have the spectral type of $(B_1,T^*, \nu^*)$ which we denote by $\sigma_{T^*,\nu^*}$. We define the spectral type of $T$ relatively to $\nu^*$ by
%$$\sigma_{T,\nu^*}=\sigma_{T^*,\nu^*}.$$
	
Another way to define the spectral type of a continuous linear operator $T$ on a Banach space $X$ is to consider the topological dynamical system $(B_{X^*}, T^*)$, where $B_{X^*}$ denotes the closed unit ball of the dual space $X^*$ equipped with the weak-$*$ topology. Since $B_{X^*}$ is compact in the weak-$*$ topology by the Banach--Alaoglu theorem and $T^*$ is weak-$*$ continuous, the Krylov--Bogolyubov theorem implies that the set $\mathcal{P}_{T^*}(B_{X^*})$ of $T^*$-invariant probability measures on $B_{X^*}$ is non-empty. Let $\nu^*$ be such an invariant measure. We obtain the spectral type of the measure-preserving system $(B_{X^*}, T^*, \nu^*)$, denoted by $\sigma_{T^*,\nu^*}$. We then define the spectral type of $T$ relative to $\nu^*$ by
	\begin{equation}
		\sigma_{T,\nu^*} = \sigma_{T^*,\nu^*}.
	\end{equation}
We notice that, by Bourgain's observation \cite{Bo1}, we have
$$\sigma_{f_v}=W^*\lim \Big|\frac{1}{\sqrt{N}}f_v(\phi \circ T^n) e^{2 i \pi n \theta}\Big|^2 d\theta, \textrm{~~for~~$\nu^*$-almost~~all~~} \phi,$$
where  $d\theta$ is the Lebesgue measure on the circle identified with $[0,1).$

\noindent{}With this observation, a connection between the two definitions can be established. We omit the details and leave the establishment of this connection as an exercise for the reader.

\noindent{}The spectral type of $T$ is said to be singular if $\sigma_{T,\nu^*}$ is singular for all $ \nu^*\in \mathcal{P}_{T^*}(B_{X^*})$, $\sigma_{T^*,\nu^*}$ is singular.
\section{The first main result and its proof.}\label{Prin}
In this section, we stat our first main result and present its proof.

\smallskip

\begin{thm}\label{main}Möbius-Sarnak process equipped with any quasi-generic probability measure $\eta \in \mathcal{I}(x)$ has a countable Lebesgue component, where $x$ stand for the privilege point of the process.
\end{thm}
For the proof, we need to observe first the following.\\
\paragraph{\bf{Key observation.}}\label{key-positive} The entropy of $\eta_M$ is positive. Indeed, if not, we get
$\Pi_{\eta_M}(S)=\mathcal{B}(\A_3)$ mod $\eta_M$, where $\mathcal{B}(\mathcal{A}_3)$ is the whole $\sigma$-algebra. We thus have that
$$
\eE(f_1|_{\Pi_{\eta_M}(S)}) = f_1=0\,,
$$
which is impossible since $\{f_n\}$ generated the whole $\sigma$-algebra. 
It follows that we have
\begin{cor}\label{Vp}Möbius-Sarnak process has a positive topological entropy.
\end{cor}
The proof of \ref{Vp} follows from the variational principle (\ref{Va1}). Let us point out also that our proof gives the following.
% we shall use the arguments of Rokhlin-Sinai (Theorem \ref{Q} combined with Proposition \ref{P} ) to prove the following.

\smallskip

\begin{prop} \label{abscont} For all quasi-generic measures $\eta \in \mathcal{I}_S(x)$
	, the spectral measure of every \linebreak $f\in {L^2(X_{\A_3},\Pi(S),\eta)}^{\perp}$ is absolutely continuous with
	respect to Lebesgue measure.
\end{prop}
We move now to the proof of our first main result.

\begin{proof}[{\bf{ of Theorem \ref{main}}}] By the definition of Möbius-Sarnak's process, let us 
	fix an quasi-generic invariant measure $\eta \in \mathcal{I}_S(x)$. We further have, by assumptions
	
	\smallskip
	
	\noindent (i) $\Pi_{\eta_M}(S)=s^{-1}(\mathcal{B}(Y))$ is the Pinsker $\sigma$-algebra of $\eta_M$,
	
	\smallskip

	\noindent (ii) $\Pi_{\eta}(S) \supset \Pi_{\eta_M}$, %where $\Pi_{\eta}(S)$ is the Pinsker $\sigma$-algebra of $\eta$.
	
	\smallskip
	
	\noindent{}Moreover, by the proprieties of conditional expectation, we have 
	
	\smallskip
	
	\noindent (iii) $\eE({f_1}|_{\Pi_{\eta}}|_{\Pi_{\eta_M}}) = \eE({f_1}|_{\Pi_{\eta_M}}) =  0$.
	
	\noindent Now note that $L^2(X, \A,\eta) = L^2(X,\Pi_{\eta_M},\eta)\oplus L^2(X,\Pi_{\eta_M},\eta)^{\perp}$
	and since $\Pi_{\eta_M} \subset \Pi_{\eta}$, we can write
	$L^2(X,\Pi_{\eta_M},\eta)^{\perp} = L^2(X,\Pi_{\eta},\eta)^{\perp}\oplus V$, where
	$$
	V = L^2(X,\Pi_{\eta},\eta)\circleddash L^2(X,\Pi_{\eta_M},\eta) \equiv
	\{f \in  L^2(X,\Pi_\eta,\eta), f \perp L^2(X,\Pi_{\eta_M},\eta) \}\,.
	$$
	Thus,
	$$
	L^2(X, \A),\eta) = L^2(X,\Pi_{\eta_M}(S),\eta)\oplus L^2(X,\Pi_{\eta},\eta)^{\perp}\oplus V\,.
	$$
	We are going to see that the spectral type of $U_T$ on %the second and the third factor of 
	the above decomposition is Lebesgue and this will established our first main result. %be a consequence of positivity of entropy of $\eta_M$.
	
	%Since entropy of $\eta$ is positive, it follows that the unitary operator $U_S$ acting on the second factor $L^2(X_{\A_3},\Pi_\eta(S),\eta)^\perp$ has countable Lebesgue spectrum by applying Rokhlin-Sinai's Proposition (\ref {P}).
	
	For that, we start by establishing the spectral type of $U_T$ on the third term of the decomposition $V$ is countable Lebesgue. Indeed, by appealing to  Rokhlin-Sinai's Theorem (\ref {Q}), one can construct a basis for $V$ that gives a countable Lebesgue spectrum. This can be done as follows. Since the entropy of \( \eta_M \) is positive, we can apply Rokhlin–Sinai Theorem (\ref{Q}) to the dynamical system $(X_,\A,\eta_M)$ to provide a sub-\(\sigma\)-algebra \( \mathcal{E} \) such that 
	\[
	T^{-1}\mathcal{E} \subset \mathcal{E}, \qquad 
	S^n \mathcal{E} \nearrow \A, 
	\quad \text{(i.e., } \bigcup_n T^n \mathcal{E} \text{ generates } \A, \quad \text{and}\ 
	T^{-n}\mathcal{E} \searrow \Pi{\eta_M}).
	\]
	
	Now apply Theorem \ref{CFS-F} by taking the algebra $\F$ to be $\F = \mathcal{E}\cap\Pi_\eta(S)$. Note that the hypothesis of this theorem holds and the Hilbert space $\ds \bigoplus_{n \in \Z}H_n$ of that theorem is just $L^2(X, \Pi_\eta, \eta)\ominus L^2(X,\Pi_{\eta_M}, \eta) = V$, (this follows from the fact that $T^n \mathcal{E}\cap \Pi_\eta  \nearrow \A\cap \Pi_\eta = \Pi_\eta$, and $T^{-n}\mathcal{E} \cap \Pi_\eta \searrow \Pi_{\eta_M}\cap \Pi_\eta = \Pi_{\eta_M}$. In other words if
	$$
	W = L^2(X_, \mathcal{E} \cap \Pi_\eta, \eta)
	\ominus L^2(X, T^{-1}\mathcal{E} \cap \Pi_{\eta_M}(S), \eta)\,.
	$$
	Then $W$ is infinite-dimensional and by taking $\{f_j, j \in J\}$ as an orthonormal basis of $W$, we see that  $\{U_T^n f_j, j \in J, n \in \Z\}$ is an orthonormal basis of $V = L^2(X,\Pi_\eta,\eta)\circleddash L^2(X,\Pi_{\eta_M},\eta)$, ($J$ is countable by separability of $L^2(X,\A,\eta),$ ). %, Theorem \ref{CFS-F} shows that restriction of $U_T$ to $V$ also has countable Lebesgue spectrum.\\
	
	To finish the proof, we notice that either the entropy $h_\eta(T)$ of $\eta$ is zero or strictly positive. 
	If $h_\eta(T)>0$, we apply once again the Rokhlin-Sinai procedure (Theorem \ref{CFS-F}), to see that the operator $U_T$ on $L^2(X,\Pi_{\eta},\eta)^{\perp}$ has countable Lebesgue spectrum, otherwise (i.e. $h_\eta(T)=0$), and in this case we have
	% then $ L^2(X,\Pi_{\eta},\eta)^{\perp}=\Big\{0\Big\}$ and 
	$$L^2(X, \A,\eta) = L^2(X,\Pi_{\eta_M},\eta)\oplus L^2(X,\Pi_{\eta_M},\eta)^{\perp}.
	$$
	It is suffice to apply once again the previous procedure the Rokhlin-Sinai procedure (Theorem \ref{CFS-F}) to $L^2(X,\Pi_{\eta_M},\eta)^{\perp}$. This 
	achieved the proof of the theorem.
\end{proof}
%\end{proof}
\section{Some Consequence and Applications.}\label{Appli}
In this section, we present several applications in number theory, culminating in our second main result.
\subsection{Applications to Number theory}
\medskip 
In Theorem \ref{SV}, the arithmetical function is the famous Möbius function $\bm$ defined by
\begin{equation}\label{Mobius}
	\bmu(n)= \begin{cases}
		1 {\rm {~if~}} n=1; \\
		\bml(n) {\rm {~if~}} n
		{\rm {~is~square-free~}} ; \\
		0 {\rm {~if~not.}}
	\end{cases}
\end{equation}
We recall that $n$ is square-free if $n$ has no factor in the subset
$\pr_2\setdef\big\{p^2/ p \in \pr\big\}$, where,
as customary, $\pr$ denotes the set of prime numbers.
The close allies of $\bmu$ is the Liouville function given by 
$$\bml(n)=(-1)^{|W(n)|},$$
where $|W(n)|$ is the length of the word $n$ in the alphabet of primes
(up to permutation) , that is,  $|W(n)|$ is the number of prime factors
of $n$ counted with multiplicities. $n$ is said to be a primitive word if $W(n)$ is the word of distinct primes (up to permutation), $n$ is also said to be square-free.

\smallskip
\begin{cor}\cite{AM}\label{AM20}The M\"{o}bius flow with respect to any $\eta \in \mathcal{I}(\bmu)$ has countable Lebesgue component. 
\end{cor}
We recall that the Möbius flow is raised by $\bmu$, that is, $X$ the closure of the orbit of $\bmu$ under the shift map. We further have
\begin{cor}\label{Elliott}\cite{AM}Any potential spectral measure of the M\"{o}bius function is absolutely continuous with respect te Lebesgue measure. 
\end{cor}
It is conjectured that the spectral measure of $\bmu$ is Lebesgue measure \cite{Eli}. As far as the author known, this conjecture still open. Corollary \ref{Elliott} has many consequences in number theory related to Matömaki-Radziwi{\l}{\l}-Tao's theorems \cite{MRT} (see \cite{AM} for more details.). According to Corollary \ref{Elliott} combined with the first main result in \cite{AD}, we conjecture the following
\begin{eqc}The densities $\rho$ of the all  potential spectral measure of the Möbius function $\bmu$ are $\log$-integrable, that is, $\log|\rho|$ is integrable. Therefore,  all  potential spectral measure of the Möbius function $\bmu$  are equivalent to Lebesgue measure.  
\end{eqc}

\noindent{}We further have
\begin{cor}\label{rigid}The M\"{o}bius disjointness conjecture holds for any rigid transformation. 
\end{cor}

\noindent{}The following extends Corollary \ref{rigid}.
\begin{cor}\label{AML26}For any operator $T$ on a Banach space with singular spectrum and any Möbius-Sarnak process with privilege point $x$, we have
	$$\frac{1}{N}\sum_{n=1}^{N}\omega_n a_n(\phi,T,v,(n_r)) \tend{N}{+\infty}{}0,$$
	where $\omega_n=f_1(T^nx)$, for any $n \in \N$.
\end{cor}
\smallskip

By considering the second definition of the spectral type of an operator, Corollary \ref{AML26} can be stated as follows
\begin{cor}\label{AML27}For any operator $T$ on a Banach space with singular spectrum and any Möbius-Sarnak process with privilege point $x$, for any continuous function $F$ on $B_{X^*}$ and any $x \in X$ we have 
	$$\frac{1}{N}\sum_{n=1}^{N}\omega_n .f_x(T^*(\phi)= 
	\frac{1}{N}\sum_{n=1}^{N}\omega_n .f_x(T \circ T) \tend{N}{+\infty}{}0,$$
	where $\omega_n=f_1(T^nx)$, for any $n \in \N$.
\end{cor}
Corollary \ref{AML26} generalizes, in the spirit of Below-Losert,  Proposition 3.10 from \cite{ALV}  due to M. Lin and the author \footnote{Therein (that is, in \cite{ALV}), among other results,  the authors proved that Veech's conjecture and Sarnak's conjecture are equivalent.}. For its proof, one need the Below-Losert's machinery which is based on the so-called affinity or geometric mean (for statisticians, it is also called Bhattacharyya coefficient). For all details about the use of this machinery, we refer to \cite{BL}, \cite{AM}.  The details of the proof are omitted and left to the reader. However, for convenience, the fundamental ingredients of the proof are recalled below; we refer to \cite{BL}, \cite{AM}, \cite{AD} for more details on the algorithm.

%The details of the proof are omitted and leave to the reader. However, for the convenience of the reader, some material from \cite{AM} is recalled below
\paragraph{Geometric mean and correlations.}\label{GM-CKMF}
Let $\mathcal{M}_1(\T)$ be a set of probability measures on
the circle $\T$ and $\eta,\nu \in \M_1(\T)$ two measures in this space. Then, there exists a probability
measure $\lambda$ such that $\eta$ and $\nu$ are absolutely continuous with respect to $\lambda$,
(take for example $\lambda=\frac{\eta+\nu}{2}$). Then the {\em geometric mean } or {\em affinity} between $\eta$ and
$\nu$ is defined by
\begin{equation}\label{affinity}
	G(\eta,\nu)=\ds \int \sqrt{\frac{d\eta}{d\lambda}. \frac{d\nu}{d\lambda}} d\lambda.
\end{equation}
This definition does not depend on $\lambda$. The geometric mean is related to the Hellinger distance which is 
defined as
\[
H(\eta,\nu)=\sqrt{2(1-G(\eta,\nu))}.
\]
As a consequence of Cauchy-Schwarz inequality we have the following,
\[
0 \leq G(\eta,\nu) \leq  1.
\]
\begin{rem}\
	\begin{enumerate}[(i)]
		\item The definition of geometric mean can be extended to any pair of positive finite measures by normalizing them.
		
		\item It is an easy exercise to see that $G(\eta,\nu)=0$ if and only if $\eta$ and $\nu$
		are mutually singular (denoted by $\eta \bot \nu$):
		this means that there exists a pair of disjoint Borel sets $A$ and $B$ such that $\eta$ is concentrated on $A$ and
		$\nu$ on $B$, (N.B. we recall that a measure $\rho$ is concentrated on a Borel set $E$ if
		$\rho(F)=0$ if and only if $ F \cap E= \emptyset$). Similarly, $G(\eta,\nu)=1$ holds if
		and only if $\eta$ and $\nu$ are equivalent. %$\eta\ll \nu$ and $\nu\ll \mu$.
		Affinity can be used to compare sequences of measures via the
		following theorem. 
	\end{enumerate}
\end{rem}

\begin{thm}[Coquet-Kamae-Mand\`es-France \cite{CMFK}] \label{coquet-france}
	Let $(P_n)$ and $(Q_n)$ be two sequences of probability measures
	on the circle, weakly converging to the probability measures $P$ and $Q$
	respectively. Then
	\begin{equation}\label{limsup}
		\limsup_{n \longrightarrow +\infty } G(P_n,Q_n) \leq G(P,Q).
	\end{equation}
\end{thm}

\begin{rem}
	As in the case of the geometric mean, this result can be generalized
	to any sequence of positive non-trivial finite measures $P_{n}$, $Q_{n}$ on a compact space converging weakly to two positive non trivial finite measures $P$ and $Q$.
\end{rem}

\noindent We aim to use the notion of affinity, together with Theorem~\ref{coquet-france}, 
to estimate the orthogonality properties of pairs of sequences for which the correlations exists (called also Wiener sequences).
To this end, we replace the sequence of Fourier coefficients 
\[
\frac{1}{n}\sum_{j=0}^{n-1} g_{j+k}\,\overline{g_{j}}
\]
by a sequence of finite positive measures on the unit circle. 

\noindent For any $g \in \mathcal{W}$, we define the sequence of measures
\begin{equation}\label{rhodef}
	d\sigma_{g,n}(x) = \rho_{g,n}(x)\,\frac{dx}{2\pi}, 
	\qquad \text{where} \quad
	\rho_{g,n}(x) = \left|\frac{1}{\sqrt{n}}\sum_{j=0}^{n-1} g_j e^{ijx}\right|^2.
\end{equation}

With this definition $\sigma_{g,n}$ defines a finite positive measure on the
circle. Moreover  we have the relation
\[
\frac{1}{n}\sum_{j=0}^{n-1} g_{j+k}\overline{g_{j}}=
\int_{0}^{2\pi } e^{-ikx} d\sigma_{g,n} (x) \ + \ \Delta_{n,k}  =
\widehat{\sigma }_{g,n} (k)+ \Delta_{n,k},
\]
where
\[
\left| \Delta_{n,k}  \right|=
\left| \frac{1}{n}\sum_{j=n-k}^{n-1} g_{j+k}\overline{g_{j}}  \right| \leq \
\frac{k}{n} \sup_j |g_{j}|^{2} \tend{n}{+\infty}{} 0.
\]
Taking the limit we have
\[
\widehat{\sigma }_{g} (k) = \lim_{n\to\infty}
\frac{1}{n}\sum_{j=0}^{n-1} g_{j+k}\overline{g_{j}}= \lim_{n\to\infty}
\widehat{\sigma }_{g,n} (k),
\]
so the sequence of measures $(\sigma_{g,n})_{n \in \N}$ converges weakly to $\sigma_{g}$.\\

\noindent As we mentioned before, it may be possible that the bounded sequence $(g_{n})_{n \in \N}$ does not belong to the
space Wiener sequences $\mathcal{W}$. But
we can always extract a subsequence $(n_r)$ such that
\[
\lim_{r \rightarrow \infty}
\frac{1}{n_r}\sum_{j=0}^{n_r-1}g_{j+k} \overline{g_{j}}
\]
exists for each $k\in \N$.
In fact, consider the sequence of finite positive measures $(\sigma_{g,n})_{n \in \N}$ on the circle defined in \eqref{rhodef}. These measures are all finite and $\sigma_{g,n} (\T)\leq  \|g\|_\infty^{2}
= \sup_j |g_{j}|^{2}$, for all $n$. Therefore they all belong to the ball $B(0,\|g\|_\infty^{2})$
centered at $0$ with radius $\|g\|_\infty^{2}$ in the  set of measures on the circle. This subset
is compact so there exists a subsequence $(n_r)$ such that the sequence of probability measures
$(\sigma_{g ,n_r})_{r \in \N}$ converges weakly to some probability measure $\sigma_{g,(n_r)}$.
The measure $\sigma_{g ,(n_r)}$ is called the spectral measure of the sequence $g$ along the
subsequence $(n_r)$.\\

\noindent We will also need the following result whose proof follows from Theorem \ref{coquet-france}.
\begin{cor}\cite{BL}\label{BeL} Let $g = (g_{n})_{n\in \N},
	h= (h_{n})_{n\in \N} \in \mathcal{W}$ be two non trivial sequences i.e.
	$\widehat{\sigma}_{g} (0)>0$ and $\widehat{\sigma}_{h} (0)>0$. Then
	\begin{equation}\label{BL-log}
		\limsup_{n\to \infty} \left | \frac{1}{n}  \sum_{j=1}^{n}g_{j}\overline{h}_{j}\right|
		\leq \sup \big\{G(\sigma_{g,(n_r)} ,\sigma_{h,(m_r)})\big\}.
	\end{equation}
	where the supremum on the right-hand side is taken over all subsequence $(n_r)$, $(m_r)$ for
	which the spectral measures exists.	
\end{cor}
We are now able to prove Corollary \ref{AML26}.
\vskip 0.2cm

\begin{proof}[{\bf{of Corollary \ref{AML26}}}]
	It is suffice to apply Corollary \ref{BeL}.
\end{proof}

\section{Topological entropy of Möbius function}\label{TopoM}
%In his seminal paper, P. Sarnack stated that the topological entropy of the Möbius flow satisfy 
%$$\frac{6}{\pi^2} \leq h_{\top}(\bmu) \leq \frac{6}{\pi^2}\log 3.$$
%He also pointed out that under Chowla conjecture, the entropy of $\bmu$ is $\frac{6}{\pi^2}\log 3.$ Later, Veech in his private Lectures \cite{V-n}\footnote{see for instance the letter of W. Veech sent to the author in \cite{AV}, in which he declare that there is only 4 people in the world how has a copy of this notes including the author. } stated without proof that the entropy of $\bmu$ is $\frac{6}{\pi^2}\log 3.$\\

%Here, using the Veech representation of the Möbius flow as skew product \cite{AM} combined with Newton's theorem \cite{N}, we unconditionally establish the following
%\begin{Theorem}\label{main2}
%\end{Theorem}
\noindent{}This section is devoted to the proof of our third and fourth results.

\noindent{}Recall that in his seminal work~\cite{Sarnak}, Sarnak showed that 
the topological entropy of the Möbius flow satisfies
%This section is devoted to the proof of our third and four results. 
%\noindent{}Let us start first by recalling that in his seminal paper, Sarnak stated that the topological entropy of the M\"{o}bius flow satisfies
\begin{equation}\label{eq:sarnak-bounds}
	\frac{6}{\pi^2} \leq h_{\mathrm{top}}(\bmu) \leq \frac{6}{\pi^2}\log 3.
\end{equation}
He further observed that, conditionally on the Chowla conjecture, the
entropy of $\bmu$ equals $\frac{6}{\pi^2}\log 3$.
Subsequently, Veech, in unpublished lecture notes~\cite{Vn},
asserted that the topological entropy of $\bmu$ is equal
to $\frac{6}{\pi^2}\log 3$.\\

%\noindent{}Here, based on the ideas of Veech 
%representation of the
%M\"{o}bius flow as a skew product~\cite{AM} with Abrambov-Rokhlin's

\noindent{} Here, by building on Veech's ideas, we establish, unconditionally, the following

\begin{thm}\label{main2} The topological entropy of the Möbius function (that is, the Möbius flow) is 
	$$h_{top}(\bmu)=\frac{6}{\pi^2}\log 3.$$
\end{thm}

%Before prooceding to the proof, we need to recall the fondamental notion admissibility due to Mirsky \cite{Ms} (see also \cite{Sarnak}).

\noindent{}Before proceeding to the proof, we recall the fundamental notion of
admissibility due to Mirsky~\cite{Mir} (see also~\cite{Sarnak}).

\smallskip

\begin{defn} A subset $A \subset \N$ is \emph{admissible} if
	the cardinality $t(p,A)$ of residue classes modulo $p^2$ represented
	by $A$, defined by
	\[
	t(p,A) \coloneqq
	\left|\bigl\{z \in \Z/p^2\Z : \exists\, n \in A,\ n \equiv z
	\pmod{p^2}\bigr\}\right|,
	\]
	satisfies
	\begin{equation}
		\label{eq:def-admissible}
		\forall\, p \in \pr,\quad t(p,A) < p^2.
	\end{equation}
	In other words, for every prime $p$, the image of $A$ under reduction
	modulo $p^2$ is a proper subset of $\Z/p^2\Z$.
\end{defn}	
\smallskip

\begin{defn} An infinite sequence
	$x = (x_n)_{n \in \N} \in X_3\equidef\{\pm 1,0\}$ is said to be \emph{admissible} if
	its \emph{support} $\{n \in \N : x_n \neq 0\}$ is admissible.
	Likewise, a finite block $x_1 \cdots x_N \in \{0,\pm 1\}^N$ is
	\emph{admissible} if $\{n \in \{1,\ldots,N\} : x_n \neq 0\}$ is
	admissible. Admissible sets in $X_2\equidef\{1,0\}$ are defined in the similar manner.
\end{defn}

\smallskip

\noindent
For each $i = 2,3$, we denote by $X_{\A_i}$ the set of all admissible
sequences in $X_i$. Since a set is admissible if and only if each of
its finite subsets is admissible, and since any translate of an
admissible set is admissible, $X_{\A_i}$ is a closed shift-invariant
subset of $X_i$, that is, a subshift. Moreover, $\bmu^2$ is an
admissible sequence, and
\[
X_{\A_3} = s^{-1}(X_{\A_2}),
\]
where $s$ denotes the squaring map. We denote by $\B(\A_i)$ the Borel
$\sigma$-algebra generated by $\A_i$, for $i = 2,3$.

\smallskip

\noindent We also denote by  $X_{\A^{*}_i},$ $i=2,3$
the subset of sequence $x$ for which the support- $\text{supp}(x)$ is infinite and let
$\Omega_j = \{\pm 1\}^{\Z_j}$, with $\Z_0 = \N \subsetneq \N \cup \{0\}$ and $\Z_1 = \Z$.  \\

\noindent Finally, For any $N \in \N$, a subset $A \subset \Z/N\Z$ will be identified with a subset $\tilde{A}$ of $[1,N].$
%For the proof, we recall first the Veech 's representation theorem in the following

\noindent We are now able to proceed to the proof of Theorem \ref{main2}.

\begin{proof}[\textbf{of Theorem \ref{main2}}].We start by defining, for any \( n > 0 \),
	\begin{align}\label{eq:par}
	A(3,n) = \{ x|_{[1,n]} = (x_1,\dots,x_n) : x \in X_{\A_3} \}.
\end{align}
	Then the cardinality of \( A(3,n) \) satisfies
	\begin{align}\label{eq:top1}
		|A(3,n)|
		&= \left| \left\{ \operatorname{supp}(x|_{[1,n]}) : x \in X_{\A_3} \right\} \right| \\
		&= \left| \left\{ A \subset [1,n] : A \text{ is admissible} \right\} \right|.
	\end{align}
	Since admissibility is preserved under the shift, we obtain, for any \( m > 0 \),
	\begin{align}\label{eq:top2}
		|A(3,n)|
		= \left| \left\{ \operatorname{supp}(x|_{[m+1,m+n]}) : x \in X_{\A_3} \right\} \right|.
	\end{align}
	Hence,
	\begin{align}\label{eq:top3}
		|A(3,n+m)| \leq |A(3,n)| \, |A(3,m)|.
	\end{align}
	By virtue of Fekete’s lemma, it follows that
	\[
	h_{\mathrm{top}}(X_{A_3}, S)
	= \lim_{n \to \infty} \frac{\log |A(3,n)|}{n}.
	\]
	%Now since $\mu^2$ is admissible, we define the following injective map
	%\begin{align}
	%\varphi : \{0,\pm 1\}^{\mathrm{supp}(\bmu^2)} & \longrightarrow X_{\mathcal A_3}\\
	%&x \longmapsto \varphi(x),
	%\end{align}
	The admissibility of $\mu^2$ allows us to define the following injection:
	\[
	\begin{aligned}
		\varphi : \Big\{0,\pm 1\Big\}^{\operatorname{supp}(\mu^2)} &\longrightarrow X_{\mathcal A_3},\\
		x &\longmapsto \varphi(x).
	\end{aligned}
	\]
	with
	$$
	\mathrm{pr}_n(\varphi(x))=
	\begin{cases}
		0 & \text{if } n \notin \mathrm{supp}(\bmu^2),\\
		x(n) & \text{otherwise}.
	\end{cases}
	$$
	%From this, we see that 
	Consequently, we obtain
	$$ h_{\mathrm{top}}(X_{A_3}, S) \geq \frac{6}{\pi^2}.\log 3,$$
	since the density of the subset ${\mathrm{supp}(\bmu^2)}$ is 
	$$\frac{6}{\pi^2}=\lim_{N \to +\infty}\frac{1}{N}\sum_{n=1}^{N}\bmu^2(n).$$
	To establish the upper bound inequality, we denote by $\pr_{2,N}$, for any $N>0$, the subset of the square of the first $N$ primes, that is, 
	$\pr_{2,N}=\big\{p^2~:~ p \in \pr_N\big\},$ where $\pr_N$ is the the subset  of the first $N$ primes. Set $q_N=\prod_{q \in \pr_{2,N}}q$. The set of square-free natural numbers is denoted by $R$, and we define $R_N$ by
	$R_N=\Big\{n \in \N~:~q\not| n, q \in \pr_{2,N}\Big\}$. Notice that $\bmu^2=\1_R,$ and following Rauzy \cite[p.99]{R}\footnote{This decomposition allows G. Rauzy to prove that $\mu^2$ is Besicovitch.}, we have
	$$\mu^2=\1_{R_N}+\mu^2-\1_{R_N}.$$
	Moreover, 
	\begin{align}\label{eq:rau1}
		\frac{1}{q_N}\sum_{n=1}^{q_N}\1_{R_N}(n)=\prod_{p \in \Pi_N}(1-\frac{1}{p^2})\equidef \zeta_N(2),
	\end{align}
	since the set of points $x \in \prod_{p \in \pr_N}\Z/p^2\Z$ none of whose coordinates is zero has cardinality $\ds \prod_{p \in \Pi_N}(p^2-1)$.
\noindent{}Let $\overline{R_N} \subset \Z/q_N\Z \cong \ds \prod_{p \in \Pi_N}\Z/p^2\Z$. 
The last isomorphism is due to the Chinese Remainder Theorem. We identify $\overline{R_N}$ with the subset $\widetilde{R_N}$ of $[1,q_n]$, we also extend its characteristic function $\1_{\widetilde{R_N}}$ to have period $q_N$. Obviously, $\mathrm{supp}(\1_{R_N})$ is exactly the set of $n\in \N$ such that $n$ is nonzero mod $p^2$ for every $p\in \pr_N$, that is, 
$$\mathrm{supp}(\1_{\widetilde{R_N}})=R_N.$$	

\noindent{}Therefore, \eqref{eq:rau1} combined with the periodicity of $\1_{R_n}$ implies
	\begin{align}\label{eq:rau2}
		\lim_{x \to +\infty}\frac{1}{x}\sum_{n=1}^{x}\1_{R_N}(n)=\zeta_N(2)
	\end{align}
since, by  \eqref{eq:rau1}, we have
$$ \big|\widetilde{R_N}\big|=\frac{q_N}{\zeta_2(N)}.$$
\noindent{}Let $\B_{3,q_N}$ be the set of "$N$-admissible" points in $\big\{0,\pm 1\big\}^{q_n}$, and $W_N^{0}$ the subset of all words of length $\big|\widetilde{R_N}\big|$ with the alphabet $\{0,\pm 1\}$, that is, $W_N^{0}=\big\{0,\pm 1\big\}^{\widetilde{R_N}}$. We thus have
$$\big|W_N^{0}\big|=3^{\frac{q_N}{\zeta_2(N)}},$$
Furthermore, for any $y \in \Z/q_n\Z$, to the set $\widetilde{R_N}+y$, we associate $W_N^{y}$ the subset of all words of length $w$ in $\{0,\pm 1\}^{\widetilde{R_N}+y}$. Then 
$$\big|W_N^{y}\big|=3^{\frac{q_N}{\zeta_2(N)}},$$
Notice further that 
$$\widetilde{R_N}+y=\big\{x \in [1,q_n]~~:~~ x \neq y \mathrm{~mod~} p^2, p \in \pr_N\big\}.$$
By the definition of $\B_{3,N}$ it follows that $z \in \B_{3,N}$ if and only if that the support $\mathrm{supp}(z)$
, of $z$ lies in at least one of the sets $R_N+y$, that is, $t(p,(\mathrm{supp}(z))) < p^2 $ with $p \in \pr_N$. It follows, from \eqref{eq:top1}, that

\begin{align}
	\big|\A_{3,q_N}\big|\leq |B(3,N)|&=\big|\bigcup_{y \in \Z/q_n\Z}W_N^{y}\big|\\
	& \leq q_N 3^{\frac{q_N}{\zeta_2(N)}}.
\end{align}
Hence,
\begin{align}
	\lim_{N \to +\infty}\frac{\log\big(\big|\A_{3,q_N}\big|\big)}{q_N} &\leq \lim_{N \to +\infty}\frac{\log\big(q_N\big)+q_N \log(3)}{q_n}\\
	&=\frac{6}{\pi^2}\log(3).
\end{align}
To finish the proof, we fix an invariant measure $\eta$ on the M\"{o}bius flow and consider the finite partition $\Cc_n$ given by the cylinder sets determined by the first $n$ coordinates of elements of $\A(3,n)$, as defined in \eqref{eq:par}. %Therefore, the atoms in $\Cc_n$ are given by $C(x_1,\cdots,x_n)$ with $\widetilde{x}=(x_1,\cdots,x_n) \in \A(3,n).$  We thus have
%$$h_{\eta}(\Cc_N,\eta)=-\sum_{\widetilde{x}\in \A(3,n)}\eta(C(\widetilde{x})) \log(C(\widetilde{x})).$$ 
%Whence, since the map $t \mapsto -\log(t)$ is convex, we get
%$$h_{\eta}(\Cc_N,\eta) \leq \log(|A(3,n)|).$$
Therefore, the atoms of $\mathcal{C}_n$ are the cylinders sets $C(x_1,\ldots,x_n)$ with 
$\widetilde{x}=(x_1,\ldots,x_n) \in \mathcal{A}(3,n)$. We thus obtain
\[
h_{\eta}(\mathcal{C}_n)
= -\sum_{\widetilde{x}\in \mathcal{A}(3,n)}
\eta\bigl(C(\widetilde{x})\bigr)\log\eta\bigl(C(\widetilde{x})\bigr).
\]
Since the map $t \mapsto -\log t$ is convex, Jensen's inequality yields
\[
h_{\eta}(\mathcal{C}_n) \leq \log\bigl(|\mathcal{A}(3,n)|\bigr).
\]
Dividing by $n$ and letting $n \to +\infty$, we obtain

$$
h_{\eta}(S, X_{\A_3}) = \lim_{n \to +\infty} \frac{h_{\eta}(\Cc_n)}{n} \leq \lim_{n \to +\infty} \frac{\log\bigl(|\A(3,n)|\bigr)}{n} = h_{\mathrm{top}}(S, X_{\A_3}).
$$

\noindent{}We complete the proof by appealing to the variational principle \eqref{Va1}. %which states
%$$h_{top}(S,X_{\A_3})=\sup_{\mu: S^*\mu=\mu}h_{\mu}(S,X_{\A_3}).$$
%	$\Cc_n$ the  the finite algebra determined by the partition, $\Q_n$ , of $\{0,\pm1\}^\N$ into the $3^n$ cylinder sets determined by the the first $n$ coordinates
\end{proof}
\begin{rem} R. Peckner established that the square-free flow is intrinsically ergodic in the sense of Weiss \cite{W2}, that is, the supremum in the variational principle is attained by a
unique invariant probability $\eta$ that we call the Peckner measure \cite{P3}.
\end{rem}
\begin{quec}
 We ask here whether the Möbius flow is also intrinsically ergodic.
\end{quec}

\noindent{}We end this section by applying Abramov–Rokhlin to compute the measure-theoretical entropy of $(X_{\A_3}, \B(\A_3),\eta_M)$. Precisely, we are going to prove the following

\begin{thm}\label{ARV} The Kolomogrov-Sinai entropy of the Möbius flow equipped with Veech-Sarnak measure is given by
	$$h_{\eta_M}(S)=\frac{6}{\pi^2}\log 2,$$
\end{thm}

%For that, we start by recalling the following Veech's representation theorem of the Möbius flow.
\noindent{}For the proof of Theorem \ref{ARV}, we need the following Veech's representation theorem.
\begin{lem}[~{\cite[Section 22, p.76]{Vn}, \cite[proof of Veech Theorem]{AM}}]\label{Vr}The Möbius flow $(X_{\A_3}),\eta_M,S)$ can be represented as the following skew product
	\begin{eqnarray}
		& \psi_0 : & X_{\A_2}\times \Omega_0\rightarrow X_{\A_2}\times \Omega_0\\
		\nonumber & & (x,\omega)\mapsto (Sx,S^{\textrm{pr}_1(x)}(\omega)).
	\end{eqnarray}
	with the natural extension of $\psi_0$ to $X_{\A_2}\times \Omega_0$ denoted by $\psi_1$. The almost everywhere isomorphism is given by the map
	$\Phi$ define `coordinate wisely'  by 
	\begin{eqnarray}
		& \Phi : & X_{\A_2}\times \Omega_0\rightarrow X_{\A_3}\\
		\nonumber & & (x,\omega)\mapsto \Phi(x,\omega) : \textrm{pr}_n(\Phi(x,\omega))=\begin{cases}
			0 &\textrm{if~~} n \not \in  \text{supp}(x)=\{n_1<n_2<\cdots<n_k<\cdots\}\\
			\textrm{pr}_k(\omega) & \textrm{if~~} n=n_k \in \text{supp}(x).
		\end{cases}
	\end{eqnarray}
	We thus have $\Phi(m \times \p)=\eta_M$, $\p$ is the Bernoulli measure $\otimes_{n \in \Z} \Big(\frac{1}{2}\delta_{-1}+\frac{1}{2}\delta_{1}\Big),$  $m$ is the Mirsky measure, and $\Phi \circ \psi_0 =S\circ \Phi$. Notice further that as a consequence of the definition of $\Phi$, we have
	\begin{eqnarray}\label{VSMA}
\Phi(X_{\A_2} \times \Omega_0)=X_{\A_3}.
	\end{eqnarray}
\end{lem} 
%W. Veech pointed out that this construction can be carried out to the setting of $\B$-free dynamical systems as it is established by K-Lema-Weiss \cite{KLW}. In fact, therein, In section 2,  the authors proved that the $\B$-free dynamical system $(X_{\B},S)$ 
%where $\B=\{b_1,b_2,\cdots\} \subset \N_{>1}=\{2,3,\cdots\},$ with $\gcd(b_i,b_j)=1$, for $i \neq j,$, $\sum_{b \in \B}\frac{1}{b}<+\infty$,  
%$S$ is the shift map and $X_{\B}={y  {0, 1}Z : y \mathrm{~is~}  \B-admissible}$ (see \cite{ALR} for the definition) is isomorph to the skew product $\widetilde{S}$ define by
%\begin{eqnarray}
%	& \widetilde{S} : & \prod_{b \in \B} \Z/b\Z \times \{0,1\}^{\Z}\rightarrow \prod_{b \in \B} \Z/b\Z \times \{0,1\}^{\Z}\\
%	\nonumber & & (\omega,x)\mapsto \big(S\omega,S^{\textrm{pr}_0(\phi(\omega))}(x)\big),
%\end{eqnarray}
%where as before $\phi$ is $\Phi$ define `coordinate wisely'  by
% \begin{eqnarray}
 %	& \phi & \prod_{b \in \B} \Z/b\Z \rightarrow \big\{0,1\big\}^\Z
 %	    \nonumber & & \omega \mapsto \phi(\omega) : \textrm{pr}_n(\phi(\omega))=\begin{cases}
 %	    	1 &\textrm{if~~} \forall i >1, \omega(i)+n \neq 0 \mathrm{mod} b_i\\
 %	    	0 & \textrm{if~~not} .
 %	    \end{cases}
%\end{eqnarray}	  
       
%We recall that the $\B$-numbers was introduced by P. Erd\H{o} and the $\B$-free dynamical systems  by el Abdalaoui-Lema\'{n}czyk and de la Rue here \cite{ALR}.

\noindent{}We shall also make use of the following formula for the entropy of a class of skew product transformations, due to Abramov–Rokhlin \cite{AR}, \cite{BHc}.
\begin{lem}[Abramov-Rokhlin~{\cite[Theorem (Abramov-Rokhlin's formula)]{BHc}}]
	\label{thm:ar} Let X and Y be
	Lebesgue spaces with measures $\mu$ and $\nu$, and $U$ on $X \times Y$ with measure $\mu \times \nu$ be defined by
	\[
	U(x, y) = \bigl(Tx,\, \mathbf{S}_x (y)\bigr),
	\]
	where $\mathbf{S}_{\bullet}=\{\mathbf{S}_x\}$ is a family of measure preserving-transformation on the probability space $(Y,\B,\nu)$. Then
	Then
	\begin{equation}
		\label{eq:ar}
		h(U) = h(T) + h(\mathbf{S}_{\bullet}|T).
	\end{equation}	
\end{lem}
Where $h(S_{\bullet}|T)$ is the fiber entropy given by 

\begin{align}\label{eq:ve1}
	h(\mathbf{S}_{\bullet}\mid T)
	&= \sup\bigl\{ h_{\nu}(\mathcal{P}, \mathbf{S}_{\bullet}\mid T)
	:\ \mathcal{P} \text{ is a finite measurable partition of } Y \bigr\}, \\
	\textrm{with~~}
	h_{\nu}(\mathcal{P}, \mathbf{S}_{\bullet}\mid T)
	&= \lim_{N \to \infty} \frac{1}{N}
	\int_{Y}
	H\!\left(
	\mathcal{P} \vee \bigvee_{j=0}^{N-1}
	\left( \mathbf{S}_{\bullet}(j,x,T)^{-1}\right)(\mathcal{P})
	\right)\, d\nu(x),
\end{align}
and 
\[
\mathbf{S}_{\bullet}(j,x,T)=
\begin{cases}
	\mathbf{S}_{T^{j-1}x}\circ \mathbf{S}_{T^{j-2}x}\circ \cdots \circ \mathbf{S}_x, & \text{if } j \geq 1,\\[4pt]
	\mathrm{Id}_Y, & \text{if } j=0.
\end{cases}
\]
\smallskip
\noindent{}We are now able to present the proof of Theorem \ref{main2}.\\

\begin{proof}[\textbf{of Theorem \ref{main2}}]. It suffice to apply Abramov-Rokhlin's formula \ref{thm:ar} with $X=X_{\A_2}, Y=\Omega_0$ equipped with $m$ and $\p$, $T=S$ and $\mathbf{S}_x=S^{\textrm{pr}_1(x)}$, where $S$ is the shift map. Therefore, by Cellarosi-Sinai-Sarnak's theorem \cite{CS}, \cite{Sarnak}, $h(S)=h_m(S)=0$ and 
	\begin{align}
		\mathbf{S}(j,x,T)
		&= \mathbf{S}_{T^{j-1}x}\circ \mathbf{S}_{T^{j-2}x}\circ \cdots \circ \mathbf{S}_x \\
		&= S^{\operatorname{pr}_j(x)} \circ S^{\operatorname{pr}_{j-1}(x)} \circ \cdots \circ S^{\operatorname{pr}_1(x)} \\
		&=\ds S^{\phi_i(x)}.
	\end{align}
	where $\ds \phi_j(x)=\sum_{i=1}^{j}\textrm{pr}_i(x)$. Now, for each $n > 0$, let $\B_n$ denote the finite algebra determined by the partition, $\pr_n$ , of $\{0,1\}^\N$ into the $2^n$ cylinder
	sets determined 
	by the the first $n$ coordinates. Let us further notice that for any $x \in X_{\A_2}$, the symbols at positions $n \notin \text{supp}(x)$ are fixed to $0$. The symbols at positions $n \in \text{supp}(x)$ are i.i.d. random variables with values $\{\pm 1\}$, each occurring with probability $1/2$. We further associate to $x$ l'atome $A_n(x) \in \pr_n(x)$. Therefore
	$$|\textrm{supp}(A_n(x))|=\phi_n(x).$$ Whence
	
	$$h_{\eta_M}(S)=h_{\eta_M}(S|X_{\A_2})=\lim \frac{1}{n}\int_{X_{\A_2}}\log(2^{\phi_n(x)})dm(x)=
	\log(2)\int_{X_{\A_2}}\textrm{pr}_1(x) m(x)=\frac{6}{\pi^2}h_{\p}(T),$$
	We finish the proof by noticing that $h_{\p}(T)=\log 2.$ \footnote{It is worth noting that Newton's formula \cite{New} remains valid, yet the hypothesis of \cite[Theorem~1]{New} is not fulfilled.}. This achieve the proof of the theorem.  
\end{proof}
\smallskip
\vskip 0.5cm
\section{Extension to the $\B$-free dynamical systems}
\noindent{}As observed by W. Veech, his representation of the square-free dynamical system as a skew product is a special case of that of the $\B$-free dynamical systems obtained by Kułaga-Przymus, Lemańczyk, and Weiss \cite{KLW}. Here, we will extend the Veech's construction for the Möbius flow to the case of $\B$-free dynamical system to obtain a skew product with a Bernoulli base in the space $\{0,-1,1\}^{\mathbb{Z}}$ endowed with the standard Bernoulli measure. %his turns out to be an essential point, since Veech's construction yields an extension of the Möbius flow to the setting of $\B$-free systems.. 
For that, by \cite[Section~2 ]{KLW}, the $\B$-free dynamical system $(X_{\mathcal{B}}, S)$, where
$$\B = \{b_1, b_2, \ldots\} \subset \N_{>1} = \{2, 3, \ldots\},$$
with $\gcd(b_i, b_j) = 1$ for $i \neq j$, and $\ds \sum_{b \in \mathcal{B}} \frac{1}{b} < +\infty$, $S$ denotes the shift map, and
$$X_{\B} = \big\{ y \in \{0,1\}^{\Z} : y \text{ is } \B\text{-admissible} \big\}$$
(see \cite{ALR} for the definition of $\B\text{-admissible}$), is isomorphic to the skew product transformation $\widetilde{S}$ defined by
\begin{eqnarray}
	\widetilde{S} : \prod_{b \in \B} \Z/b\Z \times \{0,1\}^{\Z} 
	&\longrightarrow& 
	\prod_{b \in \B} \Z/b\Z \times \{0,1\}^{\Z} \nonumber \\
	(\omega, x) &\longmapsto& \big(S\omega,\, S^{\mathrm{pr}_0(\phi(\omega))}(x)\big),
\end{eqnarray}
where $\phi$ is defined coordinate-wise by
\begin{eqnarray}
	\phi : \prod_{b \in \B} \Z/b\Z
	&\longrightarrow& 
	\{0,1\}^{\Z} \nonumber \\
	\omega &\longmapsto& \phi(\omega),
\end{eqnarray}
with
\begin{equation}
	\mathrm{pr}_n(\phi(\omega)) = 
	\begin{cases} 
		1 & \text{if } \forall\, i \geq 1, \quad \omega(i) + n \not\equiv 0 \pmod{b_i}, \\ 
		0 & \text{otherwise.}
	\end{cases}
\end{equation}
To define the Erdös-Möbius flow, we define coordinate-wise 
the map $\psi$  by
\begin{eqnarray}
	\tilde{\psi} : \prod_{b \in \B} \Z/b\Z
	&\longrightarrow& 
	\{0,-1,1\}^{\Z} \nonumber \\
	\omega &\longmapsto& \tilde{\psi}(\omega)
\end{eqnarray} 
with, 
\begin{equation}
	\mathrm{pr}_n(\psi(\omega)) = 
	\begin{cases}
		1 &\textrm{~if~} \forall\, i \geq 1, \quad \omega(i) + n \not\equiv 0 \pmod{b_i} \textrm{~~and~~} n \equiv 0 \pmod{2},\\
		-1  &\text{if } \forall\, i \geq 1, \quad \omega(i) + n \not\equiv 0 \pmod{b_i} \textrm{~~and~~} n \equiv 1\pmod{2},
		\\ 
		0 & \text{otherwise.}	
	\end{cases}	
\end{equation}
Let us put 
\begin{eqnarray}\label{Def-mu-B-f}
	\bmu_\B(n)=\psi(\bmu_{\B}^{(2)}), n \in \Z,
\end{eqnarray}
where $\bmu_\B^{(2)}$ is the characteristic function of $\B$-free numbers. Notice that 
$$s(\bmu_B)= \bmu_\B^{(2)}.$$ Denote by the closure of the orbit of $\bmu_B$ by $X_{\B,3}$. Then, in similar way as for the Möbius flow, it can be seen that  
$X_{\B,3} =\overline{\{S^n(\bmu_B), n\in \Z\}} \subset  \big\{\pm 1,0\big\}^\Z$.

\noindent{}We further have that the Möbius-$\B$-free skew product is given by 
\begin{eqnarray}
	& \psi_{\B} : & X_{\B}\times \Omega_0\rightarrow X_{\B}\times \Omega_0\\
	\nonumber & & (x,\omega)\mapsto (Sx,S^{\textrm{pr}_1(x)}(\omega)).
\end{eqnarray}
%where $\phi$ is defined coordinate-wise by\psi()
We recall that $\B$-free numbers were introduced by P.~Erd\H{o}s \footnote{
	P. Erd\H{o}s conjectured that for any $\varepsilon > 0$ there exists $N_{\mathcal{B},\varepsilon}$ such that for any $N \geq N_{\mathcal{B},\varepsilon}$ the interval $[N, N + N^\varepsilon]$ contains at least one $\mathcal{B}$-free number \cite{Er}. As far as the author knows, this conjecture is still open. In \cite{ALR}, the authors provide some evidence for this conjecture. We also refer to \cite{Za} for more details.
}, and 
$\B$-free dynamical systems  by \linebreak El~Abdalaoui, Lema\'{n}czyk, and de~la~Rue in \cite{ALR}.\\

\noindent{}The topological entropy of the $B$-free dynamical system is \cite{ALR} 
	$$h_{\mathcal{B}, \mathrm{top}}(S, X_{\mathcal{B}})=\prod_{b \in \mathcal{B}} \left(1-\frac{1}{b}\right) \log(2).$$
As for the Möbius flow, on can explore the  topological entropy of the Möbius-$\B$-free. Such topological entropy should be 
$$h_{\mathcal{B}, \mathrm{top}}(S, X_{\mathcal{B},3})=\prod_{b \in \mathcal{B}} \left(1-\frac{1}{b}\right) \log(3).$$
%We believe that the method developed here can also be used to tackle this conjecture; however, following Littlewood's practice, we prefer to state it as a question.
Although we expect this assertion to be valid in this generality, we find it more appropriate to present the problem as a question.
%\begin{que}Do we have that the topological entropy of the Möbius-$\B$-free is 
%	$$h_{\B, top}(S, X_{\B,3})=\frac{b \in \B}{1-\frac{1}{b}} \log(3).$$	
%\end{que}
\begin{que} We ask whether the topological entropy of the Möbius-$\mathcal{B}$-free system is 
	$$h_{\mathcal{B}, \mathrm{top}}(S, X_{\mathcal{B},3})=\prod_{b \in \mathcal{B}} \left(1-\frac{1}{b}\right) \log(3).$$
\end{que}
We further asked about the extension of know results about Möbius function to Möbius-$\B$-free function. In forthcoming paper, we will explore a part of this project. Notice that Sarnak's Möbius disjointness conjecture can be stated for $\bmu_{\B,3}$ as follows.
\begin{conj}For any topological dynamical system with zero topological entropy $(X,T)$, for any $f \in C(X)$, for any $x \in X$, we have
	$$\frac{1}{N}\sum_{n=1}^{N}\bmu_{\B,3}(n) f(T^nx) \tend{N}{+\infty}{}0.$$
\end{conj}

\begin{rem}Our definition of the Möbius-$\B$-free function is from dynamic. It turns out that, in Number Theory, the extension of the Möbius function was introduced 
	in the setting of $r$-free numbers, $r \geq 1$, by T. Apostol \cite[p.50-51]{Ap1},\cite{Ap2}. This extension is denoted by $\bmu_r$ and  define as follows:
	 \begin{align*}
	 	\bmu_r(1)&=1,\\
	 	\bmu_r(n)&=0 \textrm{~if~} p^{r+1}|n \textrm{~for~~some~prime~} p,\\
	 	\bmu_r(n)&=(-1)^k \textrm{~if~}  n=p^r_1\cdots p^r_k \prod_{i>k}p^{a_i}_i, 0  \leq a_i<r\\	 	
\bmu_r(n)&=1 \textrm{~~otherwise}	 
 	\end{align*}
\[
\mu_r(n)=
\begin{cases}
	1, & \text{if } n=1,\\[1ex]
	0, & \text{if } p^{r+1}\mid n \text{ for some prime } p,\\[1ex]
	(-1)^k, &
	\text{if }
	n=p_1^{r}\cdots p_k^{r}\displaystyle\prod_{i>k}p_i^{\alpha_i},
	\quad 0\le \alpha_i<r,\\[1ex]
	1, & \text{otherwise.}
\end{cases}
\]

Alternatively, 
\[
\mu_r(n)=
\begin{cases}
	0, & \text{if } p^{r+1}\mid n \text{ for some prime } p,\\[1ex]
	(-1)^{k}, & \text{otherwise},
\end{cases}
\]
where $r$ denotes the exact number of distinct prime divisors $p$ in the factorization of $n$ such that $p^r\mid n$. T. Apostol proved that $\bmu_r$ has properties similar to those $\bmu=\bmu_1$,and obviously 
$$s(\mu_r)=|\mu_r|=\1_{\{r-\textrm{free~numbers}\}}$$
For $\B$-free case, following Mirsky \cite{MC} and Carlitz \cite{Car}, one can define the Möbius $\B$-free function as follows
the solution of the following equation where 
$$1*u=\bmu_\B^{(2)},$$,
Where $*$ is the Dirichlet convolution, that is, 
\begin{eqnarray}\label{MCF}
\bmu_{B,MC}=\bmu*\bmu_\B^{(2)}.
\end{eqnarray}
We named this function Möbius-Carlitz-Mirsky function. In the case of Möbius $r$-free function of Apostol,  we get 
\[
(\mu*1_{r\text{-free}})(n)=
\begin{cases}
	\mu\!\left(\sqrt[r]{n}\right),& \text{if } n \text{ is a perfect } r\text{th power},\\[1ex]
	0,& \text{otherwise}.
\end{cases}
\]
where $\B=\{p^r , r \geq 2\}$. This formula is related to \cite[Lemma.5]{Ap2}.

We end this section by mentioning that we believe that the Erdös question of the moving interval $[N,N+N^\epsilon]$ can be talked by using our dynamical definition of the Möbius $B$-free function or the study of the dynamical property of that introduced in \eqref{MCF}. We further propose the following question inspired by a project of research proposed to me by G. Sanoli on the Sarnak's Möbius disjointness for the short interval \cite{Sano}.

\begin{que}Do we have for any dynamical system $(X,T)$ with zero topological entropy, for any continuous function $f \in C(X)$, for any $x \in X$
\[
\boxed{
	\sum_{N<n\le N+H}\mu_{\B,3}(n)f(T^n x_0)=o(H)
}
\]

as

\[
N\to\infty,\qquad H\to\infty,\qquad H=o(N),
\]
\end{que}
%Let us mention that in \cite{ALR}, an attempt was made to study Erd\H{o}s's question from a dynamical point of view. Unfortunately, now, we believe that  the definition of $\bmu_{B,MC}$ proposed therein does not appear to be well adapted to rd\H{o}s's problem.

Let us mention that in \cite{ALR}, an attempt was made to study Erd\H{o}s's question from a dynamical point of view. Unfortunately, we now believe that the definition of $\bmu_{B,MC}$ proposed therein is not well adapted to Erd\H{o}s's problem.​

%proposed therein is not well adapted to Erd\H{o}s's problem.
\end{rem}

\textbf{Acknowledgment.}
The present work was revisited during a pleasant visit to the American University in Cairo, and the first author wishes to express his sincere gratitude for the warm hospitality extended to him. He would like to thank Ahmad El-Guindy and Wafik Lotfallah for their kind invitation and for stimulating discussions on topics related to the subject of this paper. He would also like to thank Michael Lin for valuable discussions and suggestions on the subject, and Prof. R. Balasubramanian, who brought to his attention the paper by Mirsky and Carlitz.


\begin{thebibliography}{99}
	
	\bibitem{AV}
	e. H. el Abdalaoui, \emph{On Veech's proof of Sarnak's theorem on the M\"{o}bius flow}, {\it Preprint}, 2017, arXiv:1711.06326 [math.DS].
	
	\bibitem{AD}
	e. H. el Abdalaoui \& M. Disertori,  Spectral properties of the M\"{o}bius function and a random M\"{o}bius model. Stoch. Dyn. 16 (2016), no. 1, 1650005, 25 pp.
	
	\bibitem{AM}
	e. H. el Abdalaoui \& M.Nerurkar, Sarnak's M\"{o}bius disjointness for dynamical systems with singular spectrum and dissection of M\"{o}bius flow,
	{\it Preprint}, 2020,  	arXiv:2006.07646 [math.DS].
	
	\bibitem{ALR}
	e. H. el Abdalaoui, M. Lema\'{n}czyk, T. de la Rue, A dynamical point of view on the set of $\B$-free integers. Int. Math. Res. Not. IMRN 16 (2015), 7258–7286.
	
	\bibitem{ALV} 
	e. H. el Abdalaoui \& M. Lin, \emph{Operator ergodic theorems with M\"obius "weights"}, arXiv:2607.15960 [math.DS], 2024.
	
	\bibitem{AD}
	e. H. el Abdalaoui \& M. Disertori,  Spectral properties of the M\"{o}bius function and a random M\"{o}bius model. Stoch. Dyn. 16 (2016), no. 1, 1650005, 25 pp.
	
	\bibitem{AR}
	L. M. Abramov and V. A. Rokhlin, The entropy of a skew product of measure-
	preserving transformations, Amer. Math. Soc. Transl. Ser. 2,48 (1966), 255-265.
	
	\bibitem{Alex}
    Alexeyev, V. M. Existence of a bounded function of the maximal spectral type. Translated from the 1958 Russian original by A. Katok 
	Ergodic Theory Dynam. Systems 2 (1982), no. 3-4, 259-261 (1983).
	
	\bibitem{Ap1}
	T. M. Apostol, {\it Introduction to analytic number theory\/}, 
	Undergrad, Texts in Math., Springer-Verlag, New York-Heidelberg, 1976.
	
	\bibitem{Ap2}
	 T. M. Apostol, Möbius Functions of Order $k$, Pacific J. Math., 32 (1970), no. 1, 21-38.
	
	
	\bibitem{BL}
	A. Bellow and V. Losert, The weighted pointwise ergodic theorem and the individual ergodic theorem along subsequences, TAMS v. 288, no. 1 (1985), 307-345.
	
	\bibitem{BHc}
	T. Bogensch{\"u}tz and H. Crauel, Hans,The Abramov-Rokhlin formula,
	Ergodic theory and related topics III. Proceedings of the 3rd international conference, held in G\"ustrow, Germany, October 22-27, 1990,32-35,1992, Berlin etc.: Springer-Verlag.
	
	\bibitem{Car}
	L. Carlitz, A problem in additive arithmetic. II, Quarterly Journal of Mathematics (Oxford), vol. 3(1932), pp. 271-290.
	
	\bibitem{CMFK}
	J. Coquet, T. Kamae, M. Mend\`es-France, \emph{Sur la mesure spectrale de certaines suites arithm\'etiques},
	Bull. Soc. Math. France, 105 (1977), 369--384.
	
	%\bibitem{Ca}
	%JOÃO PEDRO CASIMIRO RIJO, METRIC ENTROPY AND TOPOLOGICAL ENTROPY:
	%THE VARIATIONAL PRINCIPLE,
	
	\bibitem{Bo1}
	Jean Bourgain, On the spectral type of Ornstein's class one transformations, Isr. J. Math., 84 (1993), no. 1-2, 53-63.
	
	\bibitem{CS}
	F. Cellarosi \& G. Ya. Sina\u{\i},  Ergodic properties of square-free numbers. J. Eur. Math. Soc. (JEMS) 15 (2013), no. 4, 1343-1374.
	
	\bibitem{CFS}
	I. P. Cornfeld; S. V. Fomin;Ya. G. Sina\u{\i},
	Ergodic theory. Translated from the Russian by A. B. Sosinski\u{i}
	Grundlehren Math. Wiss.,[Fundamental Principles of Mathematical Sciences] Springer-Verlag, New York, 1982. x+486 pp.
	
%	\bibitem{CMFK}
%	J. Coquet, T. Kamae, M. Mend\`es-France, \emph{Sur la mesure spectrale de certaines suites arithm\'etiques},
	%Bull. Soc. Math. France, 105 (1977), 369--384.
	
	\bibitem{Eli}
	PDTA Elliott, On the correlation of multiplicative and the sum of additive arithmetic functions, Memoirs Amer Math Soc 112 (1994), 1-88.
	
	\bibitem{Er}
	P. Erd\H{o}s, On the difference of consecutive terms of sequences defined by divisibility properties, Acta Arith. 12 (1966), 175–182.
		
	\bibitem{F4}
	D.H. Fremlin,  Measure Theory, Volume 4 (Torres Fremlin, 2001).
	
	\bibitem{G}
	E. Glasner, Ergodic theory via joinings. Mathematical Surveys and Monographs, 101. American Mathematical Society,Providence, RI, 2003.
	
	
	\bibitem{Goo}
	T.N.T. Goodman, Relating topological entropy and measure entropy,
	Bull. London Math. Soc. 3 (1971), 176-180.
	
	\bibitem{K}
	
	\bibitem{KLW}
	J. Ku\l{}aga-Przymus, M. Lema\'{n}czyk, and B. Weiss, On invariant measures for $\B$-free systems. Proc. Lond. Math. Soc. (3) 110(6) (2015), 1435–1474.
	
	\bibitem{MRT}
	K. Matom{\"a}ki, M. Radziwi{\l}{\l} and T. Tao, An averaged form of {C}howla's conjecture, Algebra Number Theory,
	Vol (9), 2015, 2167-2196
	
	\bibitem{Mir}
	L. Mirsky, Arithmetical pattern problems relating to divisibility by $r$-th powers, Proc. London Math. Soc. (2) 50, (1949). 497-508.
	
	\bibitem{MC}
	L. Mirsky, Summation formulae involving arithmetic functions. Duke Math. J.,6(2): 261-272 (June 1949). DOI: 10.1215/S0012-7094-49-01625-7
	
	\bibitem{M}
	M. Misiurewicz, A short proof of the variational principle for a $\Z_{+}^n$ action on a compact space, Ast\'erisque 40 (1976), 147–157.
	
	\bibitem{Nad}
	M. G. Nadkarni, Spectral Theory of Dynamical Systems , Birkh\"{a}user Advanced Texts : Basel Textbooks, Birkh\"{a}user Verlag, Basel, 1998.
	
	\bibitem{New}
	D. Newton, 
	On the entropy of certain classes of skew-product transformations,
	Proc. Am. Math. Soc. 21, 722-726 (1969). 
	
	\bibitem{P}
	W. Parry, \emph{Topics in ergodic theory,} Reprint of the 1981 original. Cambridge Tracts in Mathematics,
	75. Cambridge University Press, Cambridge, 2004.
	
	\bibitem{P3}
	R. Peckner, Uniqueness of the measure of maximal entropy
	for the squarefree flow,  Isr. J. Math. 210, 335–357 (2015). 
	
	\bibitem{R}
	G. Rauzy, \emph{Propri\'et\'es statistiques de suites arithm\'etiques}, PUF, 1976.
	
	\bibitem{Rh}
	V. A. Rokhlin, Lectures on the entropy theory of transformations
	with invariant measure. Usp. Mat. Nauk. 22 (1967), 3-56 (Russian)-Russian Math. Surveys 22 (1967),1-52.
	
	\bibitem{Ru}
	W. Rudin, Functional Analysis, 1991, McGraw-Hill, 2nd, +424 pp.

\bibitem{Sano}
G. Sanoli, \textit{private communication}, 2020. 
	
	\bibitem{Sarnak}
	P. Sarnak, M\"{o}bius randomness and dynamics. Not. S. Afr. Math. Soc. 43 (2012), no. 2, 89-97. 
	
	\bibitem{Vs}
	W. A. Veech. A criterion for a process to be prime. Monatsh. Math. 94(4) (1982), 335–341.
	
	\bibitem{Vn}
	W. Veech, M\"{o}bius dynamics, Lectures Notes, Spring Semester 2016, +164 pp, private.
	
	%\bibitem{Veech-preprint} W. Veech, A Conjecture Between the Chowla and
	%Sarnak Conjectures, preprint, May 6, 2016, private communication.
	
	\bibitem{W}
	P. Walters, An Introduction to Ergodic Theory, Graduate Texts in Mathematics 79, Springer, 1982.
	
	\bibitem{W2}
	Weiss, B. Intrinsically ergodic systems. Bull. Amer. Math.
	Soc. 76 1970 1266–1269.
		
	\bibitem{Z}
	A. C. Zaanen – Introduction to Operator Theory in Riesz Spaces (Springer, 1997).
	
	\bibitem{Za}
E. Alkan \& A. Zaharescu, A Survey on the Distribution of $\B$-free Numbers. Turkish Journal of Mathematics, (2006), 30 (3), 293-308.
\end{thebibliography}
\end{document}